\documentclass{amsart}[12pt]
\usepackage{txfonts}

\newtheorem{thm}{Theorem}[section]
  
  \newtheorem{lem}[thm]{Lemma}
  \newtheorem{prop}[thm]{Proposition}
  \newtheorem{cor}[thm]{Corollary}
 
\newcommand{\Hom}{\mathop{\mathrm{Hom}}\nolimits}

\newcommand{\SLF}{\mathop{\mathrm{SLF}}\nolimits}

\numberwithin{equation}{section}
\begin{document}
\title[A construction of symmetric linear functions of $\overline{U}_q (sl_2)$]{A construction of symmetric linear functions of the restricted quantum group $\overline{U}_q (sl_2)$}
\author[Y. ARIKE]{Yusuke ARIKE}
\address{Department of Pure and Applied Mathematics, Graduate School of Information science and Technology, Osaka University, Toyonaka, Osaka 560-0043, JAPAN}
\email{y-arike@cr.math.sci.osaka-u.ac.jp}
\begin{abstract}
In this paper we construct all the primitive idempotents of the restricted quantum group $\overline{U}_q (sl_2)$
and also determine the multiplication rules among a basis given by the action of generators of $\overline{U}_q (sl_2)$
to the idempotents.
By using this result we construct a basis of the space of symmetric linear functions of $\overline{U}_q (sl_2)$
and determine the decomposition of the integral of the dual of $\overline{U}_q (sl_2)$ twisted by the balancing element
to the basis of the space of symmetric linear functions.
\end{abstract}
\maketitle

\section{Introduction}
The restricted quantum group $\overline{U}_q (sl_2)$ at roots of unity has been studied in various contexts.
In \cite{FGST1} and \cite{FGST2} it is shown that 
the category of modules of the irrational vertex operator algebra $\mathcal{W}(p)$
is closely related with the category of finite-dimensional modules of $\overline{U}_q (sl_2)$
at $q = \exp(\pi \sqrt{-1} / p)$ for $p \ge 2$. 
More precisely, they determine the Grothendieck group and the center of $\overline{U}_q (sl_2)$.
Additionally they proved that the Grothendieck group, the center and 
a subspace of conformal blocks of $W(p)$,
which is invariant under the canonical $SL_2(\mathbb{Z})$ action,
are isomorphic each other.
Then it is also proved that, if $p=2$, the category of finite-dimensional modules of $\overline{U}_q (sl_2)$
and the category of modules of $\mathcal{W}(p)$ are equivalent each other.
Furthermore we can expect that the equivalence of these categories holds for any $p\ge2$.

In this paper we construct a basis of the space of symmetric linear functions of $\overline{U}_q (sl_2)$.
In order to construct a basis we determine certain basis of $\overline{U}_q (sl_2)$ which corresponding 
to indecomposable projective modules.
Since $\overline{U}_q(sl_2)$ is a finite-dimensional unimodular Hopf algebra and
the square of the antipode is inner, we can see that the space of symmetric linear functions
of $\overline{U}_q(sl_2)$ is isomorphic to the center by \cite{R}.
For $\overline{U}_q(sl_2)$ the center is $(3p-1)$-dimensional (see \cite{FGST1}) so we see that the space of
symmetric linear functions of $\overline{U}_q(sl_2)$ is also  $(3p-1)$-dimensional.
It also follows from \cite{R} that the linear functions given by the action of the balancing element of $\overline{U}_q(sl_2)$ to the
left and right integrals of the dual Hopf algebra of $\overline{U}_q(sl_2)$
are symmetric.
We determine the decomposition of this linear function
into the basis of the space of symmetric linear functions.

Our motivation to study symmetric linear functions comes from conformal field theories.
In the study of conformal field theory associated with a vertex operator algebra(VOA),
the representation theory of vertex operator algebras plays an important role.
In fact, the theory for any rational vertex operator algebra with the factorization property
is established over the projective line in \cite{NT}. 
This theory has been generalized to the higher genus case;
in particular, for an elliptic curve the space of conformal blocks with a vacuum module $V$
is nothing but the space of formal characters of modules for $V$ (cf. \cite{Z}).
Therefore its dimension coincides with the number of simple modules for $V$ up to
isomorphisms.

On the other hand, the theory for irrational VOAs is rather difficult. 
For example,
under the same finiteness conditions, it is shown that conformal blocks
are finite dimensional (see \cite{Mi} ). 
In this case the dimension is greater than the number of simple modules. 

There is an example of irrational conformal field theory
which is called logarithmic conformal field theory. 
A typical example of the theory is a VOA $\mathcal{W}(p)$, whose conformal blocks involve logarithmic function 
of modulus $q$ (recall that no logarithmic terms appear in rational cases).
In this example determining the dimension of conformal blocks is rather difficult.

By the discussions given in \cite{Mi} and \cite{NaTsu},
we can observe that the space of conformal blocks is isomorphic to the space of the symmetric linear functions
of a finite dimensional algebra whose category of modules is equivalent to the category
of $V$-modules (also see \cite{MNT}).
This observation naturally indicates that the space of conformal blocks is closely related with
the space of the symmetric linear functions of $\overline{U}_q(sl_2)$. 

This paper is organized as follows.
In section $2$ we review symmetric linear functions of an associative algebra and Integrals of Hopf algebras.
We also review the relationship given in \cite{R} between the space of symmetric linear functions of a finite-dimensional Hopf algebra and its center.

In section $3$ we review the definition of $\overline{U}_q(sl_2)$.
And we describe left and right integrals of $\overline{U}_q(sl_2)$, which turn out to be equal to each other,
hence $\overline{U}_q(sl_2)$ is unimodular.
And left and right integrals of the dual of $\overline{U}_q(sl_2)$ are determined.
Moreover we recall that the square of the antipode of $\overline{U}_q(sl_2)$ is inner.
The results above can be found in \cite{FGST1}.

In section $4$, by using the structure of projective modules of $\overline{U}_q(sl_2)$ in \cite{FGST1} and the result in \cite{R},
 we construct all indecomposable left ideals of $\overline{U}_q(sl_2)$
and we also construct all primitive idempotents of $\overline{U}_q(sl_2)$.

In section $5$ we give a basis of $\overline{U}_q(sl_2)$ which is the basis of indecomposable projective modules in $\overline{U}_q(sl_2)$
given by the actions of the generators of $\overline{U}_q(sl_2)$ to the primitive idempotents.
And we determine the multiplication rules among the basis.
By using the multiplication rules we construct a basis of symmetric linear functions of $\overline{U}_q(sl_2)$.
Moreover we determine the decomposition of the linear function given by 
the action of the balancing element to the left and right integrals into the basis.

\section{Preliminaries}
In this paper we will always work over the complex number field $\mathbb{C}$.
For any vector space $V$ we denote the space $\Hom_{\mathbb{C}}(V, \mathbb{C})$ by $V^*$. 
\subsection{Symmetric linear functions}
Let $A$ be a finite-dimensional associative algebra.
A symmetric linear function $\varphi$ is an element of $A^*$ which satisfies $\varphi(ab) = \varphi(ba)$
for all $a, b \in A$.
We denote the subspace of symmetric linear functions of $A^*$ by $\SLF(A)$.
If $A$ is a finite-dimensional Hopf algebra,
the space $\SLF(A)$ is equal to the space of cocommutative elements of $A^*$ (\cite{R}).

\subsection{Integrals and the square of antipode of Hopf algebras}
Let A be a finite-dimensional Hopf algebra with coproduct $\Delta$, counit $\varepsilon$ and antipode $S$.
Each of elements of the subspaces
\begin{align}
&\mathcal{L}_A = \{ \Lambda \in A | a \Lambda = \epsilon (a) \Lambda \ \text{for all} \ a \in A  \}, \notag\\
&\mathcal{R}_A = \{ \Lambda \in A |  \Lambda a = \epsilon (a) \Lambda \ \text{for all} \ a \in A  \}, \notag
\end{align}
is called a left integral and a right integral of $A$, respectively.
If $\mathcal{L}_A \not= \{ 0\}$ (respectively $\mathcal{R}_A \not= \{ 0 \}$) the space
$\mathcal{L}_A$ (respectively $\mathcal{R}_A$) is one-dimensional (cf. \cite{A}).
Similarly a left (respectively, right) integral of the dual Hopf algebra $H^*$ is an element $\lambda \in A^*$ which satisfies $p\lambda = p(1)\lambda$
(respectively, $\lambda p = p(1)\lambda$) for all $p \in A^*$.
Equivalently we can see
\begin{align}
&\mathcal{L}_{A^*} = \{ \lambda \in A^* | (1 \otimes \lambda) \Delta(x) = \lambda(x) \ \text{for all} \ x \in A \}, \notag\\
&\mathcal{R}_{A^*} = \{ \lambda \in A^* | (\lambda \otimes 1) \Delta(x) = \lambda(x) \ \text{for all} \ x \in A\}. \notag
\end{align}
If $\mathcal{L}_A = \mathcal{R}_A$ the Hopf algebra $A$ is called unimodular.
\begin{prop}[\cite{R}] \label{prop:center-integ}
Let $A$ be a finite-dimensional unimodular Hopf algebra with antipode $S$.
Suppose that $\lambda$ is the left integral of $H^*$ and that $\mu$ is the right integral of $H^*$. 
Then
\begin{enumerate}
\item
$\lambda (ab) = \lambda (bS^2(a))$,
\item
$\mu (ab) = \mu(S^2(b)a)$.
\end{enumerate}
\end{prop}

The square of the antipode is called inner if there exists an invertible element $t$
such that $S^2(x) = txt^{-1}$ for all $x \in A$.
\begin{prop}[\cite{R}] \label{prop:center-yaho}
Let $A$ be a finite-dimensional unimodular Hopf algebra with antipode $S$.
 If $S^2$ is inner,
the center $Z(A)$ of $A$ is isomorphic to $\SLF(A)$ as vector spaces.
\end{prop}

Denote by $\rightharpoonup$ and $\leftharpoonup$ the left and right actions of $A$ on $A^*$ defined by
\begin{align}
a \rightharpoonup p (b) = p(ba), \ p \leftharpoonup a (b) = p(ab), \notag 
\end{align}
for $a, b \in A$ and $p \in A^*$.

\section{The restricted quantum group $\overline{U}_q (sl_2)$}

\subsection{Definition}
Let $p \ge 2$ be a positive integer and $q = \exp(\pi \sqrt{-1} / p)$.
The restricted quantum group $\overline{U}_q (sl_2)$ is a Hopf algebra over $\mathbb{C}$
generated by $E, F, K$ and $K^{-1}$ with the relations
\begin{align}
&KK^{-1}=K^{-1}K=1, \notag\\
&KEK^{-1} = q^2 E, \ KFK^{-1} = q^{-2} F, \ [E, F] = \frac{K-K^{-1}}{q - q^{-1}},\notag\\
&E^p=F^p=0, \ K^{2p} = 1, \notag
\end{align}
as an algebra.
The coproduct $\Delta$, counit $\varepsilon$ and antipode $S$ are given by
\begin{align}
&\Delta (E) = 1 \otimes E + E \otimes K, \ \Delta (F) = K^{-1} \otimes F + F \otimes 1, \ \Delta (K) = K \otimes K, \notag\\
&\epsilon (E) = \epsilon (F) = 0,  \ \epsilon (K) = 1, \notag\\
&S(E) = -EK^{-1}, \ S(F) = -KF, \ S(K) = K^{-1}. \notag
\end{align}

\begin{lem}\label{lem:dim}
The $2p^3$ elements $E^m F^n K^{\ell}$, where $0 \le m, n \le p-1$ and $0 \le \ell \le 2p-1$,
form a basis of $\overline{U}_q (sl_2)$ as a vector space.
\end{lem}

For $m, n \in \mathbb{Z}$ we use the standard notation
\begin{align}
&[n] = \frac{q^n - q^{-1}}{q - q^{-1}}, \notag\\
&[n]! = [n] [n-1] \dotsm [2] [1], \ [0]! = 1, \notag\\
&\begin{bmatrix}
m \\ n 
\end{bmatrix}
= \frac{[m]!}{[n]![m-n]!} \ \text{for} \  n \ge 0  \ \text{and} \  m - n \ge 0. \notag
\end{align}

We can write down the coproduct of the basis of $\overline{U}_q (sl_2)$ as in Lemma \ref{lem:dim}
by using induction.
\begin{lem}\label{lem:coproduct}
\begin{align}
\Delta( E^m F^n K^{\ell}) =& \sum_{r=0}^{m} \sum_{s=0}^{n} q^{r(m-r) + s(n-s) - 2rs}
\begin{bmatrix}
m \\ n
\end{bmatrix}
\begin{bmatrix}
n \\ s
\end{bmatrix} \notag\\
& \qquad \times E^r F^{n-s} K^{-s+\ell} \otimes E^{m-r} F^s K^{r+\ell}. \notag
\end{align}
\end{lem}

\subsection{The integrals}
The integral of $\overline{U}_q (sl_2)$ and the right integral of $\overline{U}_q (sl_2)^*$ is given in \cite{FGST1}.
The basis of the space of left integrals is given by 
\begin{align}
E^{p-1}F^{p-1} \sum_{\ell = 0}^{2p-1} K^{\ell}, \notag
\end{align}
which also belongs to the space of right integrals.
Therefore we can see that $\overline{U}_q (sl_2)$ is unimodular.

Define the elements in $\overline{U}_q (sl_2)^*$ by
\begin{align}
&\lambda (E^m F^n K^{\ell}) = \delta_{m, p-1} \delta_{n, p-1} \delta_{\ell, p-1}, \notag\\
&\mu (E^m F^n K^{\ell}) = \delta_{m, p-1} \delta_{n, p-1} \delta_{\ell, p+1}. \notag
\end{align}

\begin{prop}\label{prop:integ}
Each of the spaces of left integrals and right integrals of the dual Hopf algebra of $\overline{U}_q (sl_2)$
is spanned by $\lambda$ and $\mu$ respectively.
\end{prop}
\begin{proof}[\rm Proof]
It follows from Lemma \ref{lem:coproduct}.
\end{proof}

\subsection{The square of the antipode}
$\overline{U}_q (sl_2)$ is not quasitriangular but there exists the Hopf algebra $\bar{D}$ which contains 
$\overline{U}_q (sl_2)$ as a subalgebra and which is a ribbon quasitriangular Hopf algebra (see \cite{FGST1}).
It is also shown that the Drinfeld element and the ribbon element 
of $\bar{D}$ belong to $\overline{U}_q (sl_2)$ in \cite{FGST1}.
Thus the balancing element of $\bar{D}$ is also in $\overline{U}_q (sl_2)$.

\begin{prop}[\cite{FGST1}] \label{prop:square}
The square of the antipode of $\overline{U}_q (sl_2)$ is inner, in particular,
$S^2(x)=gxg^{-1}$ for all $x \in \overline{U}_q (sl_2)$
where $g=K^{p+1}$.
\end{prop}

Then $g^{-1} \rightharpoonup \lambda$ and $\mu \leftharpoonup g$ is in $\SLF(\overline{U}_q (sl_2))$
by Proposition \ref{prop:center-integ}.
Note that 
\begin{align}
g^{-1} \rightharpoonup \lambda (E^m F^n K^{\ell}) = \mu \leftharpoonup g (E^m F^n K^{\ell}) = \delta_{m, p-1} \delta_{n, p-1} \delta_{\ell, 0}. \label{eq:integrals}
\end{align}

\subsection{Irreducible modules}
The irreducible modules $\mathcal{X}_s^{\alpha}$ are labeled by $\alpha = \pm$ and $1 \le s \le p$.
The irreducible module $\mathcal{X}_s^{\alpha}$ is spanned by weight vectors $a_n^{\alpha}(s)$,
$1 \le n \le s-1$ with the action of $\overline{U}_q (sl_2)$ defined by
\begin{align}
K a_n^{\alpha}(s) =& \alpha q^{s-1-2n}a_n^{\alpha}(s), \notag\\
E a_n^{\alpha}(s) =& \alpha [n] [s-n]  a_{n-1}^{\alpha}(s),\notag\\
F a_n^{\alpha}(s) =&   a_{n+1}^{\alpha}(s), \notag
\end{align}
where $a_{-1}^{\alpha}(s) = a_{s}^{\alpha}(s) = 0$.

\subsection{Casimir element}
The Casimir element of $\overline{U}_q (sl_2)$ is given by 
\begin{align}
C = EF + \frac{q^{-1}K + q K^{-1}}{(q-q^{-1})^2} \in Z(\overline{U}_q (sl_2)). \label{def:casimir}
\end{align}

\begin{prop}[\cite{FGST1}]
The minimal polynomial relation of Casimir element is
\begin{align}
\Phi_p(x) = (x-\beta_0)(x-\beta_p) \prod_{s=1}^{p-1} (x - \beta_s)^2 \label{eq:minimal}
\end{align}
where $\beta_s = \frac{q^s+q^{-j}}{(q-q^{-1})^2}$.
\end{prop}

This relation gives the decomposition of $\overline{U}_q (sl_2)$ into its subalgebras by Casimir element.
\begin{align}
\overline{U}_q (sl_2) = \bigoplus_{s=0}^{p} Q_s, \label{directsum:1}
\end{align} 
where $Q_s$ for $0 \le s \le p$ is generalized eigenspace of eigenvalue $\beta_s$.

\section{Idempotents of $\overline{U}_q (sl_2)$}

In this section we construct primitive idempotens of $\overline{U}_q (sl_2)$ by referring to the structure of projective modules (see \cite{FGST1}).
We construct all projective modules in $\overline{U}_q (sl_2)$ by $E$, $F$ and $K^{\pm1}$.

\subsection{Indecomposable modules in $\overline{U}_q (sl_2)$}
In order to construct projective modules in $\overline{U}_q (sl_2)$
first we construct the module whose socle is the irreducible module $\mathcal{X}_s^{\alpha}$.

Now we introduce the useful lemma.
\begin{lem}[\cite{K}]\label{lem:commutation}
For $1 \leq m \leq p - 1$, The following relations hold in $\overline{U}_q(sl_2)$:
\begin{align}
[E,F^m]  & =   [m]F^{m-1}\frac{q^{-(m - 1)}K - q^{m - 1}K^{-1}}{q - q^{-1}} \notag\\
         & =   [m]\frac{q^{m - 1}K - q^{- (m - 1)}K^{-1}}{q - q^{-1}}F^{m - 1} \notag\\
[E^m,F]  & =   [m]E^{m-1}\frac{q^{m - 1}K - q^{-(m - 1)}K^{-1}}{q - q^{-1}} \notag\\
         & =   [m]\frac{q^{-(m - 1)}K - q^{ m - 1}K^{-1}}{q - q^{-1}}E^{m - 1} \notag
\end{align}
\end{lem}
Then we have a generalization of this lemma (cf. \cite{L}, \cite{S}).
\begin{lem}\label{lem:genera}
\begin{align}
[E^r, F^s] = \sum_{i=1}^{\min(r,s)} E^{r-i}F^{s-i} f^{r,s}_i(K), \notag
\end{align}
where $f^{r,s}_i(z) \in \mathbb{C}[z, z^{-1}]$.
\end{lem} 

For $1 \le s, t \le p$ and $\alpha=\pm$, we set
\begin{align}
v^{\alpha}(s, t) = \sum_{\ell=0}^{2p-1} \left(\alpha q^{-(s-2t+1)}\right)^{\ell} K^{\ell}. \notag
\end{align}
Then we can see that $K v^{\alpha}(s, t) = \alpha q^{s-2t+1} v^{\alpha}(s, t)$.

Define
\begin{align}
a_0^{\alpha}(s, t) = E^{p-1}F^{p-t}v^{\alpha}(s, t).
\end{align}
This element in $\overline{U}_q (sl_2)$ is a highest weight vector of highest weight $\alpha q^s-1$.

\begin{lem}\label{lem:1}
Set $a_n^{\alpha}(s, t) = F^{n} a_0^{\alpha}(s, t)$.
\begin{align}
a_n^{\alpha}(s, t) = \sum_{\ell =0}^{n} \lambda_{\ell, n}^{\alpha}(s) E^{p-1-\ell}F^{p-t+n-\ell} v^{\alpha}(s, t) \notag
\end{align}
and
\begin{align}
&\lambda_{\ell, n}^{\alpha}(s) = \lambda_{\ell, n-1}^{\alpha}(s) + \alpha[n][s-2n+\ell]\lambda_{\ell-1, n-1}^{\alpha}(s),\quad \text{for} 
\quad 1 \le \ell \le n-1, \notag\\
&\lambda_{n, n}^{\alpha}(s) = \alpha [n][s-n]\lambda_{n-1, n-1}^{\alpha}(s), \notag\\
&\lambda_{0, n}^{\alpha}(s) = \lambda_{0, n-1}^{\alpha}(s). \notag
\end{align}
In particular $\lambda_{n, n}^{\alpha}(s)=\prod_{i=1}^{n}(\alpha[n][s-n])$.
\end{lem}
\begin{proof}[\rm Proof]
We use induction on $n$.
For $n=0$, the statement is clear and $\lambda_{0, 0}^{\alpha}(s)=1$.

Suppose $n > 0$.
By using Lemma \ref{lem:commutation},
we can see 
\begin{align}
a_n^{\alpha}(s, t) &= F a_{n-1}^{\alpha}(s, t) \notag\\
&=\sum_{\ell =0}^{n-1} \lambda_{\ell, n-1}^{\alpha}(s) FE^{p-1-\ell}F^{p-t+n-1-\ell} v^{\alpha}(s, t) \notag\\
&= \sum_{\ell =0}^{n-1} \lambda_{\ell, n-1}^{\alpha}(s) E^{p-1-\ell}F^{p-t+n-\ell} v^{\alpha}(s, t) \notag\\
&+ \sum_{\ell =0}^{n-1} \alpha [\ell+1][s+\ell-2n+1]\lambda_{\ell, n-1}^{\alpha}(s) E^{p-2-\ell}F^{p-t+n-1-\ell} v^{\alpha}(s, t) \notag\\
&= \lambda_{0, n-1}^{\alpha}(s) E^{p-1}F^{p-t+n}v^{\alpha}(s, t) \notag\\
& \quad + \sum_{\ell =1}^{n-1} \left( \lambda_{\ell, n-1}^{\alpha}(s)+\alpha [\ell][s+\ell-2n]\lambda_{\ell, n-1}^{\alpha}(s)\right) E^{p-1-\ell}F^{p-t+n-\ell} v^{\alpha}(s, t) \notag\\
&\qquad + \alpha [n][s-n] E^{p-1-n}\lambda_{n-1, n-1}(s)F^{p-t}  v^{\alpha}(s, t). \notag
\end{align}
Thus we have the lemma.
\end{proof}

By this lemma we can see that $a_n^{\alpha}(s, t)$ is non-zero for $0 \le n \le s-1$ so the elements
$a_n^{\alpha}(p, t)$ for $0 \le n \le p-1$
are non-zero.
Then Lemma \ref{lem:commutation} shows
\begin{align}
K a_n^{\alpha}(s,t) =& \alpha q^{s-1-2n}a_n^{\alpha}(s,t), \label{eq:aa1}\\
E a_n^{\alpha}(s,t) =& \alpha [n] [s-n]  a_{n-1}^{\alpha}(s,t),\label{eq:aa2}\\
F a_n^{\alpha} (s,t)=&   a_{n+1}^{\alpha}(s,t), \label{eq:aa3}
\end{align}
and $E a_s^{\alpha}(s,t) = 0$.
For $s=p$ it is clear that the space $\mathcal{X}_p^{\alpha}(t)$ spanned by $a_n^{\alpha}(p,t)$, $0 \le n \le p-1$
is isomorphic to the irreducible module $\mathcal{X}_p^{\alpha}$.
It is expected that $a_s^{\alpha}(s, t)$ is zero for $1 \le s \le p-1$ but it is  hard to prove by direct calculation.
We consider the element which is sent to $a_0^{\alpha}(s, t)$ by the action of $F$.

\begin{lem}\label{lem:2}
For $1 \le s \le p-1$ and $1 \le t \le s$, we have
\begin{align}
a_0^{\alpha}(s, t) = F \sum_{n=1}^{p-s} \mu _{n}^{\alpha}(s) E^{p-n}F^{p-t-n}v^{\alpha}(s, t) \notag
\end{align}
where $\mu _{n}^{\alpha}(s) = \prod_{k=p-s-(n-1)}^{p-s-1}(-\alpha[k][p-s-k])$.
\end{lem}
\begin{proof}[\rm Proof]
Direct calculation and Lemma \ref{lem:commutation} prove this lemma.
\end{proof}
Set 
\begin{align}
x_k^{\alpha}(s, t) = \frac{E^{p-s-k-1}}{\prod_{i=k+1}^{p-s-1}(-\alpha[i][p-s-i])} \sum_{n=1}^{p-s} \mu _{n}^{\alpha}(s) E^{p-n}F^{p-t-n}v^{\alpha}(s, t). \label{eq:x1}
\end{align}
Then we can easily see that
\begin{align}
x_0^{\alpha}(s, t) = \frac{\mu _{p-s}^{\alpha}(s) E^{p-1}F^{s-t}v^{\alpha}(s, t)}{\prod_{i=1}^{p-s-1}(-\alpha[i][p-s-i])}
= E^{p-1} F^{s-t} v^{\alpha}(s, t), \notag
\end{align}
so $x_k^{\alpha}(s, t)$ is non-zero and $x_0^{\alpha}(s, t)$ is a highest weight vector of highest weight $-\alpha q^{p-s-1}$.
Then, by Lemma \ref{lem:commutation} and Lemma \ref{lem:2}, we have
\begin{align}
&K x_k^{\alpha}(s, t) = -\alpha q^{p-s-1-2k}x_k^{\alpha}(s, t), \label{eq:xa1} \\
&Ex_{k}^{\alpha}(s, t) = \begin{cases}
                    -\alpha [k][p-s-k]x_{k-1}^{\alpha}(s, t), & 1 \le k \le p-s-1, \\
                    0, & k=0,
                    \end{cases}\label{eq:xa2} \\
&F x_{k}^{\alpha}(s, t) =  \begin{cases}
                   x_{k+1}^{\alpha}(s, t), & 0 \le k \le p-s-2, \\
                   a_0^{\alpha}(s, t), & k=p-s-1.
                   \end{cases}\label{eq:xa3}
\end{align}
By the relations above we obtain
\begin{align}
a_s^{\alpha}(s, t) = F^{s}a_0^{\alpha}(s, t) = F^{s + p-s} x_0^{\alpha}(s, t) = 0 \notag
\end{align}
for $1 \le s \le p-1$ and $1 \le t \le s$.

\begin{prop}\label{prop:1}
For $1 \le s \le p$ and $1 \le t \le s$
the space $\mathcal{X}_s^{\alpha}(t)$
spanned by the vectors of the form $a_n^{\alpha}(s, t)$, $0 \le n \le s-1$
is isomorphic to the irreducible module $\mathcal{X}_s^{\alpha}$.
\end{prop}

Set
\begin{align}
&b_n^{\alpha} (s, t) = F^n \sum_{n=1}^{p-s} \mu _{n}^{\alpha}(s) E^{p-n-1}F^{p-t-n}v^{\alpha}(s, t), \label{eq:b1} \\
&y_k^{\alpha}(s, t) = F^{s+k} b_0^{\alpha} (s, t) \label{eq:y1}
\end{align}
for $0 \le n \le s-1$, $0 \le k \le p-s-1$, and $\alpha=\pm$.
By Lemma \ref{lem:genera} we have $b_n^{\alpha} (s, t) \not= 0$ and $y_k^{\alpha}(s, t) \not= 0$ for $0 \le n \le s-1$ and $0 \le k \le p-s-1$.
Then the direct calculation and Lemma \ref{lem:commutation} shows the following relations:

\begin{align}
&Kb_n^{\alpha}(s,t) = \alpha q^{s-1-2n}b_n^{\alpha}(s, t),  \label{eq:ba1}\\
&Eb_n^{\alpha}(s, t) = \begin{cases}
                    \alpha [n][s-n] b_{n-1}^{\alpha}(s,t) + a_{n-1}^{\alpha}(s,t), & 1 \le n \le s-1, \\
                    x_{p-s-1}^{\alpha}(s, t), & n=0,
                    \end{cases}\label{eq:ba2}\\
&Fb_n^{\alpha}(s, t) = \begin{cases}
                 b_{n+1}^{\alpha}(s, t), & 0 \le n \le s-2, \\
                 y_0^{\alpha}(s, t), & n = s-1,
                 \end{cases}\label{eq:ba3}\\
&Ky_{k}^{\alpha}(s, t) = -\alpha q^{p-s-1-2k}y_{k}^{\alpha}(s, t) \label{eq:ya1}, \\
&Ey_{k}^{\alpha}(s,t) = \begin{cases}
                    -\alpha [k][p-s-k]y_{k-1}^{\alpha}(s, t), & 1 \le k \le p-s-1, \\
                    a_{s-1}^{\alpha}(s, t), & k=0,
                    \end{cases}\label{eq:ya2}\\ 
&Fy_{k}^{\alpha}(s,t) = \begin{cases}
                   y_{k+1}^{\alpha}(s,t), & 0 \le k \le p-s-2, \\
                   0, & k=p-s-1.
                   \end{cases}\label{eq:ya3}
\end{align}
Let $\mathcal{P}^{\alpha}_s(t)$, $1 \le s \le p-1$ and $1 \le t \le s$,
be the space spanned by the elements of the form
\begin{align}
b_n^{\alpha}(s,t), x_k^{\alpha}(s, t), y_{k}^{\alpha}(s,t), a_n^{\alpha}(s,t), \notag
\end{align}
for $0 \le n \le s-1$ and $0 \le k \le p-s-1$.
By the relations (\ref{eq:aa1})-(\ref{eq:aa3}), (\ref{eq:xa1})-(\ref{eq:xa3}) and (\ref{eq:ba1})-(\ref{eq:ya3}),
the $2p$-dimensional space $\mathcal{P}^{\alpha}_s(t)$ is an indecomposable left $\overline{U}_q(sl_2)$-module.
Note that the modules $\mathcal{P}^{\alpha}_s(t)$ for $1 \le t \le s$ are isomorphic to each other.

\begin{prop}\label{prop:projective}
The $2p$-dimensional indecomposable left module $\mathcal{P}^{\alpha}_s$,
$1 \le s \le p-1$ and $\alpha = \pm$, is spanned by weight vectors
\begin{align}
b_n^{\alpha}(s), x_k^{\alpha}(s), y_{k}^{\alpha}(s), a_n^{\alpha}(s), \notag
\end{align}
for $0 \le n \le s-1$ and $0 \le k \le p-s-1$ with left actions defined by
\begin{align}
&Kb_n^{\alpha}(s) = \alpha q^{s-1-2n}b_n^{\alpha}(s),  \notag\\
&Eb_n^{\alpha}(s) = \begin{cases}
                    \alpha [n][s-n] b_{n-1}^{\alpha}(s) + a_{n-1}^{\alpha}(s), & 1 \le n \le s-1, \\
                    x_{p-s-1}^{\alpha}(s), & n=0,
                    \end{cases}\notag\\
&Fb_n^{\alpha}(s) = \begin{cases}
                 b_{n+1}^{\alpha}(s), & 0 \le n \le s-2, \\
                 y_0^{\alpha}(s), & n = s-1,
                 \end{cases}\notag
\end{align}
\begin{align}
&K x_k^{\alpha}(s) = -\alpha q^{p-s-1-2k}x_k^{\alpha}(s),  \notag\\
&Ex_{k}^{\alpha}(s) = \begin{cases}
                    -\alpha [k][p-s-k]x_{k-1}^{\alpha}(s), & 1 \le k \le p-s-1, \\
                    0, & k=0,
                    \end{cases} \notag\\
&F x_{k}^{\alpha}(s) =  \begin{cases}
                   x_{k+1}^{\alpha}(s), & 0 \le k \le p-s-2, \\
                   a_0^{\alpha}(s), & k=p-s-1,
                   \end{cases}\notag\\                 
&Ky_{k}^{\alpha}(s) = -\alpha q^{p-s-1-2k}y_{k}^{\alpha}(s), \notag\\
&Ey_{k}^{\alpha}(s) = \begin{cases}
                    -\alpha [k][p-s-k]y_{k-1}^{\alpha}(s), & 1 \le k \le p-s-1, \\
                    a_{s-1}^{\alpha}(s), & k=0,
                    \end{cases}\notag\\ 
&Fy_{k}^{\alpha}(s) = \begin{cases}
                   y_{k+1}^{\alpha}(s), & 0 \le k \le p-s-2, \\
                   0, & k=p-s-1,
                   \end{cases}\notag\\
&K a_n^{\alpha}(s) = \alpha q^{s-1-2n}a_n^{\alpha}(s), \notag\\
&E a_n^{\alpha}(s) =\begin{cases} \alpha [n] [s-n]  a_{n-1}^{\alpha}(s), & 1 \le n \le s-1, \\
                                    0, & n=0,
                      \end{cases} \notag\\
&F a_n^{\alpha} (s)=   \begin{cases} a_{n+1}^{\alpha}(s), & 0 \le n \le s-2, \\ 
                                        0, & n=s-1. \\
                          \end{cases}\notag
\end{align}
\end{prop}

Note that $(C-\beta_s)^2$ vanishes on the module $\mathcal{P}^{+}_{s}$ and $\mathcal{P}^{-}_{p-s}$.
Hence there are inclusion maps $\mathcal{P}^{+}_s \to Q_s$ and $\mathcal{P}^{-}_{p-s} \to Q_s$ for $1 \le s \le p-1$.
Since each of $(C-\beta_0)$ and $(C-\beta_p)$ vanishes on $\mathcal{X}_p^-$ and $\mathcal{X}_p^+$, respectively
there are inclusion maps $\mathcal{X}_p^- \to Q_0$ and $\mathcal{X}_p^+ \to Q_p$(see \cite{FGST1}).

\subsection{Primitive idempotents of $\overline{U}_q (sl_2)$}
First we show the irreducible modules $\mathcal{X}_p^{\alpha}(t)$
contain primitive idempotens of $\overline{U}_q (sl_2)$.
\begin{prop}\label{prop:irr-idem}
Set
\begin{align}
e^{\alpha}(p, t) = \frac{1}{2p\prod_{i=1}^{p-1}\left(\alpha[i][p-i]\right)} a_{t-1}^{\alpha}(p,t) \in \mathcal{X}_s^{\alpha}(t).
\end{align}
Then
\begin{align}
e^{\alpha}(p, t_1) e^{\alpha}(p, t_2) = \begin{cases}
                                                 e^{\alpha} (p, t_2), & t_1=t_2, \\
                                                 0, & \text{otherwise}.
                                                 \end{cases}
                                                 \end{align}
In particular each 
$e^{+}(p, t)$ and $e^{-}(p, t)$, $1 \le t \le p$, is primitive idempotent of $Q_p$ and $Q_0$, respectively.
\end{prop}
\begin{proof}[\rm Proof]
It is clear that $e^{\alpha}(p, t)$ generates the irreducible module $\mathcal{X}_p^{\alpha}$.
By the left action of $K$ on the irreducible module,
\begin{align}
v^{\alpha}(p, t_1) a_{t_1-1}^{\alpha}(p,t_2) =\begin{cases}
                                                2p a_{t_2-1}^{\alpha}(p,t_2) , & t_1=t_2, \\
                                                 0, & \text{otherwise}.
                                                 \end{cases} \notag
                                                 \end{align}
By the action of $\overline{U}_q(sl_2)$ on the irreducible module
\begin{align}
a_{t-1}^{\alpha}(p,t)^2 &= 2p F^{t-1} E^{p-1}F^{p-t} a_{t-1}^{\alpha}(p,t) \notag\\
&= 2p\prod_{i=1}^{p-1}(\alpha[i][p-i]) F^{t-1}a_{0}^{\alpha}(p,t) \notag\\
&= 2p\prod_{i=1}^{p-1}(\alpha[i][p-i]) a_{t-1}^{\alpha}(p,t). \notag
\end{align} 
\end{proof}

To find other primitive idempotents of $\overline{U}_q(sl_2)$ the following lemma is useful.
\begin{lem}\label{lem:eigenvalue}
Let $\varphi$ be an element in $\overline{U}_q (sl_2)$ with weight $q^{s-1-2n}$ for $0 \le n \le s-1$
and $\psi$ be an element in $\overline{U}_q (sl_2)$ with weight $-q^{p-s-1-2k}$ for $0 \le k \le p-s-1$.

\begin{align}
&v^+(s,t) \varphi = \begin{cases}
                   2p \varphi, & n=t-1,\\
                   0, & \text{otherwise},
                   \end{cases}\notag\\
&v^+(s,t) \psi = 0, \notag\\
&v^-(p-s, u) \varphi = 0, \notag\\
&v^-(p-s, u) \psi = \begin{cases}
                    2p \psi, & k=u-1, \\
                    0, & \text{otherwise},
                    \end{cases}\notag
\end{align}
for $1 \le s \le p-1$, $1 \le t \le s$ and $1 \le u \le p-s$.
\end{lem}
\begin{proof}[\rm Proof]
It follows from direct calculation.
\end{proof}
For $1 \le s \le p-1$ and $1 \le t \le s$ set
\begin{align}
e^{\alpha}(s,t) = \frac{1}{\gamma^{\alpha}(s)}\left(b_{t-1}^{\alpha}(s,t) - \frac{\delta^{\alpha}(s)}{\gamma^{\alpha}(s)} a_{t-1}^{\alpha}(s,t)\right)
\in \mathcal{P}_s^{\alpha}(t)
\end{align}
where
\begin{align}
&\gamma^{\alpha}(s) = 2p \prod_{m=1}^{p-s-1}(-\alpha[m][p-s-m]) \prod_{i=1}^{s-1} (\alpha[i][s-i]), \label{eq:gamma}\\
&\delta^{\alpha}(s) = 2p \prod_{m=1}^{p-s-1}(-\alpha[m][p-s-m])\sum_{j=1}^{s-1}\prod_{\genfrac{}{}{0pt}{}{ k =1}{k \neq j}}^{s-1}(\alpha[k][s-k])  \label{eq:delta}\\
& \qquad \qquad \qquad + 2p \prod_{i=1}^{s-1} (\alpha[i][s-i])\sum_{n=1}^{p-s-1}\prod_{\genfrac{}{}{0pt}{}{k=1}{k \neq n}}^{p-s-1}(-\alpha[k][p-s-k]). \label{eq:00}\notag
\end{align}
Note that $\gamma^{+}(s) = \gamma^{-}(p-s)$ and $\delta^{+}(s) = \delta^{-}(p-s)$.
Then we see that $e^{\alpha}(s,t)$ generates the module $\mathcal{P}_s^{\alpha}$.

Using Lemma \ref{lem:eigenvalue} and the action of $\overline{U}_q (sl_2)$ we have the following:
\begin{prop}\label{prop:idempotent}
The elements $e^{+}(s, t)$ for $1 \le t \le s$ and 
$e^{-}(p-s, u)$ for $1 \le u \le p-s$ are mutually orthogonal primitive idempotents
of $Q_s$ for $1 \le s \le p-1$.
\end{prop}
\begin{proof}[\rm Proof]
Using Lemma \ref{lem:eigenvalue}, we have
\begin{align}
&\left(e^{+}(s, t)\right)^2 \notag\\
&=  \frac{F^{t-1}}{\left(\gamma^{+}(s)\right)^2} 
\left(b_{0}^{+}(s,t) - \frac{\delta^{+}(s)}{\gamma^{+}(s)} a_{0}^{+}(s,t)\right) 
 \left(b_{t-1}^+(s,t) - \frac{\delta^{+}(s)}{\gamma^{+}(s)} a_{t-1}^{+}(s,t)\right)\notag\\
&=2p\frac{F^{t-1}}{(\gamma^{+}(s))^2}\left(\sum_{i=1}^{p-s} \mu _{i}^{+}(s) E^{p-i-1}F^{p-t-i}-
\frac{\delta^+(s)}{\gamma^+(s)}E^{p-1}F^{p-t}\right)
\left(b_{t-1}^+(s,t) - \frac{\delta^{+}(s)}{\gamma^{+}(s)} a_{t-1}^{+}(s,t)\right). \notag
\end{align}
By the action of $\overline{U}_q(sl_2)$ on $\mathcal{P}_s^+$,
\begin{align}
&\left(\sum_{i=1}^{p-s} \mu _{i}^{+}(s) E^{p-i-1}F^{p-t-i}\right) \left(b_{t-1}^+(s,t) - \frac{\delta^{+}(s)}{\gamma^{+}(s)} a_{t-1}^{+}(s,t)\right) \notag\\
&= \mu_{p-s}^+(s) \prod_{i=1}^{s-1} ([i][s-i]) \left(b_{0}^+(s,t) - \frac{\delta^{+}(s)}{\gamma^{+}(s)} a_{0}^{+}(s,t)\right)
+ \mu_{p-s}^+(s) \sum_{i=1}^{s-1}\prod_{\genfrac{}{}{0pt}{}{ j =1}{j \neq i}}^{s-1}([j][s-j])a_{0}^{+}(s,t) \notag\\
&\qquad+ \prod_{i=1}^{s-1} ([i][s-i])\sum_{i=1}^{p-s-1} \mu_{i}^+(s) \prod_{j=1}^{p-s-1-i}
(-[j][p-s-j])  a_0^+(s, t) \notag\\
&= \frac{\gamma^+(s)}{2p} \left(b_{0}^+(s,t) - \frac{\delta^{+}(s)}{\gamma^{+}(s)} a_{0}^{+}(s,t)\right) + \frac{\delta^+(s)}{2p} a_0^+(s, t), \notag
\end{align}
and
\begin{align}
&E^{p-1}F^{p-t} \left(b_{t-1}^+(s,t) - \frac{\delta^{+}(s)}{\gamma^{+}(s)} a_{t-1}^{+}(s,t)\right)
= \prod_{i=1}^{p-s-1}
(-[i][p-s-i]) \prod_{j=1}^{s-1} ([j][s-j]) a_{0}^{+}(s,t)\notag\\
&= \frac{\gamma^+(s)}{2p} a_{0}^{+}(s,t), \notag
\end{align}
since $\mu_i^+(s) = \prod_{j=p-s-(i-1)}^{p-s-1}(-[j][p-s-j])$.
Hence we have
\begin{align}
\left(e^{+}(s, t)\right)^2  &= 2p\frac{F^{t-1}}{(\gamma^{+}(s))^2} \left( \frac{\gamma^+(s)}{2p} \left(b_{0}^+(s,t) - \frac{\delta^{+}(s)}{\gamma^{+}(s)} a_{0}^{+}(s,t)\right) +  \frac{\delta^+(s)}{2p} a_0^+(s, t) - \frac{\delta^{+}(s)}{\gamma^{+}(s)}\frac{\gamma^+(s)}{2p} a_{0}^{+}(s,t)\right) \notag\\
&= e^{+}(s, t). \notag
\end{align}

By Lemma \ref{lem:eigenvalue}
these idempotents are mutually orthogonal.
\end{proof}

\begin{cor}
The modules $\mathcal{X}_p^{\alpha}$ and $\mathcal{P}_s^{\alpha}$ for $\alpha=\pm$ and $1 \le s \le p-1$
are indecomposable projective modules.
\end{cor}

It is clear that
\begin{align}
\overline{U}_q (sl_2) \supseteq  \bigoplus_{s=1}^{p-1}
\bigoplus_{t=1}^s (\mathcal{P}_s^{+}(t) \oplus  \mathcal{P}_s^{-}(t)) \oplus 
\bigoplus_{t=1}^p (\mathcal{X}_p^+(t) \oplus \mathcal{X}_p^-(t)).
\end{align}
Then the dimension of right hand side is $2p^3$ so the above inclusion is an equality.

\section{Symmetric linear functions of $\overline{U}_q (sl_2)$}

\subsection{Basis and multiplication of $Q_s$}
For $s=0, p$ we can choose a basis of $Q_s$ as follows:
\begin{align}
A_n^{\alpha}(p, t) = \frac{1}{2p\prod_{i=1}^{p-1}(\alpha[i][p-i])} a_{n}^{\alpha}(p,t),
 \ \text{for} \ 0 \le n \le p-1 \ \text{and} \ 1 \le t \le p. 
\end{align}
By similar argument in the proof of Proposition \ref{prop:irr-idem} we have
\begin{align}
A_{m}^{\alpha}(p, t_1) A_n^{\alpha}(p, t_2) = \begin{cases}
                                                 A_m^{\alpha} (p, t_2), & n=t_1-1 \\
                                                 0, & \text{otherwise}.
                                                 \end{cases} \label{eq:irr-mult}
                                       \end{align}
Then the algebras $Q_s$ for $s=0, p$ are isomorphic to $M_p(\mathbb{C})$ as algebras.

For $1 \le s \le p-1$ the following elements define a basis of $Q_s$.
\begin{align}
&B_n^+(s, t) := \frac{1}{\gamma_s}\left(b_n^+(s,t) - \frac{\delta_s}{\gamma_s}a_n^+(s,t)\right), \label{def:1}\\
&X_k^+(s,t) := \frac{1}{\gamma_s} x_k^+(s,t) = \frac{E^{p-s-k}}{\prod_{i=k+1}^{p-s-1}(-[i][p-s-i])}B_0^+(s, t), \label{def:2}\\
&Y_k^+(s,t) := \frac{1}{\gamma_s} y_k^+(s,t)=F^{s+k} B_0^+(s, t), \label{def:3}\\
&A_n^+(s,t) := \frac{1}{\gamma_s} a_n^+(s,t) = F^{n+1}E B_0^+(s, t), \label{def:4}\\
&B_k^-(p-s, u) := \frac{1}{\gamma_s} \left(b_k^-(p-s,u) - \frac{\delta_s}{\gamma_s}a_k^-(p-s,u)\right), \label{def:5}\\
&X_n^-(p-s, u) := \frac{1}{\gamma_s} x_n^-(p-s,u) = \frac{E^{s-k}}{\prod_{i=n+1}^{s-1}([i][s-i])} B_0^-(p-s, u), \label{def:6}\\
&Y_n^-(p-s, u) := \frac{1}{\gamma_s} y_n^-(p-s,u) = F^{p-s+n} B_0^-(p-s, u), \label{def:7}\\
&A_k^-(p-s, u) := \frac{1}{\gamma_s} a_k^-(p-s,u) = F^{k+1}E B_0^-(p-s, u)\label{def:8}, 
\end{align}
for $1 \le t \le s$, $1 \le u \le p-s$, $0 \le n \le s-1$ and $0 \le k \le p-s-1$
where
\begin{align}
&\gamma_s = \gamma^{+}(s) = \gamma^{-}(p-s), \notag\\
&\delta_s = \delta^+(s) = \delta^-(p-s) .\notag
\end{align}
Similar argument in the proof of Proposition \ref{prop:idempotent} the following holds.
\begin{align}
&B_m^+(s, t_1) B_n^+(s, t_2) = \begin{cases}
                              B_m^+(s, t_2), & n=t_1-1,\\
                              0, & \text{otherwise},
                              \end{cases}\label{eq:1}\\
&B_m^+(s, t_1) X_k^+(s, t_2) = 0, \label{eq:2}\\
&B_m^+(s, t_1) Y_k^+(s, t_2) = 0, \label{eq:3}\\ 
&B_m^+(s, t_1) A_n^+(s, t_2) = \begin{cases}
                              A_m^+(s, t_2), & n=t_1-1,\\
                              0, & \text{otherwise},
                              \end{cases}\label{eq:4}
                              \end{align}
\begin{align}
&B_m^+(s, t_1) B_k^-(p-s, u) = 0, \label{eq:5}\\
&B_m^+(s, t_1) X_n^-(p-s, u) = \begin{cases}
                              X_m^-(p-s, u), & n=t_1-1,\\
                              0, & \text{otherwise},
                              \end{cases}\label{eq:6}\\
&B_m^+(s, t_1) Y_n^-(p-s, u) = \begin{cases}
                              Y_m^-(p-s, u), & n=t_1-1,\\
                              0, & \text{otherwise},
                              \end{cases}\label{eq:7}\\
&B_m^+(s, t_1) A_k^-(p-s, u) = 0,\label{eq:8}\\
&B_k^-(p-s, u_1) B_n^+(s, t) = 0, \label{eq:9}\\
&B_k^-(p-s, u_1) X_{\ell}^+(s, t) = \begin{cases}
                                    X_k^+(s,t), & \ell = u_1-1, \\
                                    0, & \text{otherwise},
                                    \end{cases}\label{eq:10}\\
&B_k^-(p-s, u_1) Y_{\ell}^+(s, t) = \begin{cases}
                                    Y_k^+(s, t), & \ell=u_1-1, \\
                                     0, & \text{otherwise},
                                    \end{cases} \label{eq:11}\\ 
&B_k^-(p-s, u_1) A_n^+(s, t) = 0, \label{eq:12}\\
&B_k^-(p-s, u_1) B_{\ell}^-(p-s, u_2) = \begin{cases}
                                         B_k^-(p-s, u_2), & \ell= u_1-1, \\
                                         0, & \text{otherwise},
                                         \end{cases} \label{eq:13}\\
&B_k^-(p-s, u_1) X_{n}^-(p-s, u_2) = 0, \label{eq:14}\\
&B_k^-(p-s, u_1) Y_n^-(p-s, u_2) = 0 \label{eq:15} \\
&B_k^-(p-s, u_1) A_{\ell}^-(p-s, u_2) = \begin{cases}
                                         A_k^-(p-s, u_2), & \ell=u_1-1,\\
                                         0, & \text{otherwise}.
                                         \end{cases}\label{eq:16}
\end{align}

By \eqref{def:1} - \eqref{def:8} and \eqref{eq:1}-\eqref{eq:16},
we can determine the multiplication table among the basis:
\small{
\begin{center}
\begin{tabular}{c|cccc}
\multicolumn{5}{c}{} \\
$x \backslash y$    & $B_{t_1-1}^+(s,t_2)$   & $X_{u-1}^+(s,t_2)$ & $Y_{u-1}^+(s,t_2)$ & $A_{t_1-1}^+(s,t_2)$ \\ \hline
$B_n^+(s,t_1)$        & $B_n^+(s, t_2)$   & $0$     & $0$     & $A_n^+(s, t_2)$      \\ 
$X_k^+(s,t_1)$        & $X_k^+(s,t_2)$   & $0$     & $0$     & $0$          \\ 
$Y_k^+(s,t_1)$        & $Y_k^+(s,t_2)$   & $0$     & $0$     & $0$          \\ 
$A_n^+(s,t_1)$        & $A_n^+(s,t_2)$   & $0$     & $0$     & $0$          \\ 
$B_k^-(p-s,u)$      & $0$       & $X_k^+(s,t_2)$ & $Y_k^+(s,t_2)$ & $0$          \\ 
$X_n^-(p-s,u)$      & $0$       & $0$     & $A_n^+(s,t_2)$ & $0$          \\ 
$Y_n^-(p-s,u)$      & $0$       & $A_n^+(s,t_2)$ & $0$     & $0$          \\ 
$A_k^-(p-s,u)$        & $0$       & $0$     & $0$     & $0$          \\ 
\end{tabular}

\begin{tabular}{c|cccc}
\multicolumn{5}{c}{} \\
$x \backslash y$    & $B_{u_2-1}^-(p-s,u_2)$   & $X_{t-1}^-(p-s,u_2)$ & $Y_{t-1}^-(p-s,u_2)$ & $A_{u_2-1}^-(p-s,u_2)$ \\ \hline
$B_n^+(s,t)$        & $0$   & $X_n^-(p-s,u_2)$     & $Y_n^-(p-s,u_2)$     & $0$      \\ 
$X_k^+(s,t)$        & $0$   & $0$                  & $A_k^-(p-s,u_2)$     & $0$          \\ 
$Y_k^+(s,t)$        & $0$   & $A_k^-(p-s,u_2)$     & $0$     & $0$          \\ 
$A_n^+(s,t)$        & $0$   & $0$                  & $0$     & $0$          \\ 
$B_k^-(p-s,u_1)$    & $B_k^-(p-s,u_2)$       & $0$ & $0$ & $A_k^-(p-s,u_2)$          \\ 
$X_n^-(p-s,u_1)$    & $X_n^-(p-s,u_2)$       & $0$ & $0$ & $0$          \\ 
$Y_n^-(p-s,u_1)$    & $Y_n^-(p-s,u_2)$       & $0$ & $0$     & $0$          \\ 
$A_k^-(p-s,u_1)$    & $A_k^-(p-s,u_2)$       & $0$ & $0$     & $0$          \\ 
\end{tabular}
\end{center}
}
for $1 \le t, t_1, t_2, n+1 \le s$ and $1 \le u, u_1, u_2, k+1 \le p-s$
and other multiplications among the basis are all zero.
\subsection{Symmetric linear functions}
$\overline{U}_q(sl_2)$ is unimodular and the square of the antipode is inner.
So the dimension of the space of symmetric linear functions of $\overline{U}_q(sl_2)$ is equal to that of the center of $\overline{U}_q(sl_2)$.
The structure of the center of $\overline{U}_q(sl_2)$ is given in \cite{FGST1}.
\begin{prop}[\cite{FGST1}]
The center of $\overline{U}_q(sl_2)$ is $(3p-1)$-dimensional.
\end{prop}
By Proposition \ref{prop:center-yaho}, we have:
\begin{cor}\label{cor:sym}
The space of symmetric functions of $\overline{U}_q(sl_2)$ is $(3p-1)$-dimensional.
\end{cor}

Since both $Q_0$ and $Q_p$ are isomorphic to the matrix algebra $M_p(\mathbb{C})$
each of the trace of $Q_0$ and $Q_p$ respectively give symmetric linear functions.
The trace $T_0$ and $T_p$ is given by
\begin{align}
T_0(x) = \sum_{t=1}^p \psi_{t-1, t}^-(x), \notag\\
T_p(x)=\sum_{t=1}^{p} \psi_{t-1,t}^+(x),\notag
\end{align}
for $x= \sum_{t=1}^p \sum_{n=0}^{p-1} \psi^{\pm}_{n, t}(x) A_{n}^{\pm}(p, t) \in Q_0 \ \text{or} \ Q_p$.

Next we determine the symmetric linear functions of $Q_s$ for $1 \le s \le p-1$.
Let us write $x \in Q_s$ by the basis of $Q_s$:
\begin{align}
&x = \sum_{t=1}^{s} \Bigl\{ \sum_{n=0}^{s-1} \left( \varphi_{n,t}^+(x) B_n^+(s,t) + \psi_{n,t}^+(x) A_n^+(s,t) \right) \notag\\
     &\qquad \qquad + \sum_{k=0}^{p-s-1} \left( \xi_{k,t}^+(x) X_k^+(s,t) + \zeta_{k,t}^+(x) Y_k^+(s,t) \right) \Bigr\} \notag\\
    &+ \sum_{u=1}^{p-s} \Bigl\{ \sum_{k=0}^{p-s-1} \left( \varphi_{k,u}^-(x) B_k^-(p-s,u) + \psi_{k,u}^-(x) A_k^-(p-s,u) \right) \notag\\
     &\qquad \qquad + \sum_{n=0}^{s-1} \left( \xi_{n,u}^-(x) X_n^+(p-s,u) + \zeta_{n,u}^-(x) Y_n^-(p-s,u) \right)  \Bigr\}, \label{eq:ex} 
\end{align}
where all coefficients are complex numbers.
Now we define the linear functions 
\begin{align}
&T_s^+(x) = \sum_{t=1}^{s} \varphi_{t-1,t}^+(x), \label{eq:sym1} \\
&T_s^-(x) = \sum_{u=1}^{p-s} \varphi_{u-1,u}^-(x), \label{eq:sym2} \\
&G_s(x) = \sum_{t=1}^{s}  \psi_{t-1,t}^+(x) + \sum_{u=1}^{p-s} \psi_{u-1,u}^-(x). \label{eq:sym3}
\end{align}
By the multiplications among the basis, we have 
\begin{align}
&T_s^{+}(xy)= \sum_{t=1}^{s}\sum_{n=0}^{s-1} \varphi_{t-1, n+1}^+(x) \varphi_{n, t}^+(y), \notag\\
&T_s^{-}(xy)= \sum_{u=1}^{p-s}\sum_{k=0}^{p-s-1} \varphi_{u-1, k+1}^-(x) \varphi_{k, u}^-(y), \notag\\
&G_s(xy) = \sum_{t=1}^s \left\{ \sum_{n=0}^{s-1} \left(\psi_{t-1, n+1}^+(x) \varphi_{n, t}^{+}(y) + \varphi_{t-1, n+1}^+(x)
\psi_{n, t}^+(y)\right)
+ \sum_{k=0}^{p-s-1}\left(\zeta_{t-1, k+1}^-(x) \xi_{k, t}^+(y) + \xi_{t-1, k+1}^-(x) \zeta_{k, t}^+(y)\right) \right\} \notag\\
&+ \sum_{u=1}^{p-s} \left\{ \sum_{k=0}^{p-s-1}\left(\psi_{u-1, k+1}^-(x) \varphi_{k, u}^-(y)
+ \varphi_{u-1,k+1}^-(y) \psi_{k, u}^-(y)\right) 
+ \sum_{n=0}^{s-1}\left(\zeta_{u-1, n+1}^+(x) \xi_{n,u}^-(y) + \xi_{u-1, n+1}^+(x) \zeta_{n,u}^-(y)\right) \right\}. \notag
\end{align}
This shows that these linear functions are symmetric.
Then we have the following result:
\begin{thm}\label{}
The linear functions $T_s^{\pm}$ and $G_s$ for $1 \le s \le p-1$ $T_0$ and $T_p$ are symmetric linear functions
in particular, these linear functions form a basis of $\SLF(\overline{U}_q(sl_2))$.
\end{thm}

\subsection{Integrals and symmetric linear functions}
The linear function $g^{-1} \rightharpoonup \lambda = \mu \leftharpoonup g$ is symmetric
so this symmetric linear function can be written by 
\begin{align}
g^{-1} \rightharpoonup \lambda = \alpha_0 T_0 + \alpha_p T_p + \sum_{s=1}^{p-1} \left( \alpha_s^+ T_s^+ + \alpha_s^- T_s^- + \beta_s G_s \right), \label{eq:symm-integ} 
\end{align}
for some complex numbers $\alpha_0, \alpha_p, \alpha_s^{\pm}$ and $\beta_s$.
Recall that 
\begin{align}
A_n^{\alpha}(p, t) = \frac{F^n}{2p \prod_{i=1}^{p-1}(\alpha[i][p-i])} E^{p-1}F^{p-t} v^{\alpha} (p, t). \notag
\end{align}
Then,
by Lemma \ref{lem:1}, we have
\begin{align}
g^{-1} \rightharpoonup \lambda \left(A_n^{\alpha}(p, t)\right) 
= \begin{cases}
 \frac{1}{2p \prod_{i=1}^{p-1}(\alpha[i][p-i])}, & n=t-1, \\
 0, & \text{otherwise}.
 \end{cases} \notag
\end{align}
Since 
\begin{align}
&A_n^{+}(s, t) = \frac{F^n}{\gamma_s} E^{p-1}F^{p-t} v^{+} (s, t), \notag\\
&B_n^+(s, t) = \frac{F^n}{\gamma_s} \left( \sum_{\ell=1}^{p-s} \mu _{\ell}^{+}(s) E^{p-\ell-1}F^{p-t-\ell} - 
\frac{\delta_s}{\gamma_s} E^{p-1}F^{p-t}\right) v^{+}(s, t), \notag 
\end{align}
we can see that
\begin{align}
&g^{-1} \rightharpoonup \lambda (A_{n}^{+}(s, t)) = \begin{cases}
                                                    \frac{1}{\gamma_s}, & n=t-1,\\
                                                    0, & \text{otherwise},
                                                    \end{cases} \notag\\
&g^{-1} \rightharpoonup \lambda (B_n^+(s, t)) = \begin{cases}
                                                -\frac{\delta_s}{(\gamma_s)^2}, & n=t-1,\\
                                                 0, & \text{otherwise},
                                                 \end{cases} \notag
\end{align}
by Lemma \ref{lem:1}.
Similarly we also have
\begin{align}
&g^{-1} \rightharpoonup \lambda (A_{k}^{-}(p-s, u)) = \begin{cases}
                                                    \frac{1}{\gamma_s}, & k=u-1,\\
                                                    0, & \text{otherwise},
                                                    \end{cases} \notag\\
&g^{-1} \rightharpoonup \lambda (B_k^-(p-s, u)) = \begin{cases}
                                                -\frac{\delta_s}{(\gamma_s)^2}, & u=k-1,\\
                                                 0, & \text{otherwise}. 
                                                 \end{cases} \notag
\end{align}
Consequently we have
\begin{align}
&\alpha_0 = \frac{1}{2p \prod_{i=1}^{p-1}(-[i][p-i])}, \notag\\
&\alpha_p = \frac{1}{2p \prod_{i=1}^{p-1}([i][p-i])} , \notag\\
&\alpha_s^{\pm} = -\frac{\delta_s}{(\gamma_s)^2}, \notag\\
&\beta_s = \frac{1}{\gamma_s}. \notag
\end{align}
Now we can see that
\begin{align}
\beta_s &=  \frac{1}{2p \prod_{i=1}^{s-1} [i][s-i] \prod_{i=1}^{p-s-1}(-[i][p-s-i])} \notag\\
&= \frac{(-1)^{p-s-1} [s]^2 }{2p([p-1]!)^2} \notag\\
&= \frac{(-1)^{p-s-1}  }{2p^3}[s]^2 (2 \sin \frac{\pi}{p})^{2(p-1)}, \notag
\end{align}
since $[p-1]! = \prod_{\ell = 1}^{p-1} \sin \frac{\ell \pi}{p} / \sin^{p-1} \frac{\pi}{p}$ and $\prod_{\ell = 1}^{p-1} \sin \frac{\ell \pi}{p} = p/2^{p-1}$. 
Similarly we have
\begin{align}
&\alpha_p = \frac{ 1 }{2p^3} \left( 2 \sin \frac{\pi}{p} \right)^{2(p-1)}, \notag\\
&\alpha_0 = \frac{(-1)^{p-1}  }{2p^3}\left( 2 \sin \frac{\pi}{p} \right)^{2(p-1)}. \notag 
\end{align}
By \eqref{eq:00}, we see
\begin{align}
\alpha_s^{\pm} = -\frac{\delta_s}{(\gamma_s)^2} = - \beta_s \left( \sum_{\ell=1}^{s-1} \frac{1}{[\ell][s- \ell]} - \sum_{\ell=1}^{p-s-1} \frac{1}{[\ell][p-s-\ell]}\right). \notag
\end{align}

Acknowledgment.
The author thanks to K. Nagatomo and Y. Sakane for many helpful discussion.
He also thanks to J. Murakami for information about the relation between integrals and symmetric linear functions
and fruitful comments.


\begin{thebibliography}{FGST2}
\bibitem{A}
E. Abe: Hopf Algebras, Cambridge Univ. Press, (1980).
\bibitem{FGST1}
B. L. Feigin, A. M. Gainutdinov, A. M. Semikhatov and I. Yu. Tipunin:
\textit{Modular group representations and fusion in logarithmic conformal field theories and in the quantum group center},
Commun. Math. Phys. \textbf{265} (2006), 47-93.
\bibitem{FGST2}
B. L. Feigin, A. M. Gainutdinov, A. M. Semikhatov and I. Yu. Tipunin:
\textit{Kazhdan-Lustig correspondence for the representation category of the triplet $W$-algebra in logarithmic CFT},
Theor. Math. Phys. \textbf{148} (2006), 1210-1235.
\bibitem{K}
C. Kassel:
Quantum Groups,
Graduate Text in Mathematics. \textbf{155}, Springer-Verlag, (1995).
\bibitem{L}
G. Lusztig: \textit{Quantum deformations of certain simple modules over enveloping algebras},
Adv. in Math. \textbf{70}, (1988), no. 2, 237-249.
\bibitem{MNT}
A. Matsuo, K. Nagatomo and A. Tsuchiya:
\textit{Quasi-finite algebras graded by Hamiltonian and vertex operator algebras},
math.QA/050571.
\bibitem{Mi}
M. Miyamoto:
\textit{Modular invariance of vertex operator algebras satisfying $C_2$-cofiniteness},
Duke. Math. J. \textbf{122} (2004), 51-91.
\bibitem{NaTsu}
K. Nagatomo:
\textit{Private communication}.
\bibitem{NT}
K. Nagatomo and A. Tsuchiya:
\textit{Conformal field theories associated to regular chiral vertex operator algebras ‡T: theories over the projective line},
Duke. Math. J. \textbf{128} (2005), 393-471.
\bibitem{R}
D. E. Radford:
\textit{The trace function and Hopf algebras},
J. Algebra. \textbf{163} (1994), 583-622.
\bibitem{S}
R. Suter:
\textit{Modules over $\mathfrak{U}_q(\mathfrak{sl}_2)$},
Commun. Math. Phys. \textbf{163} (1994), 359-393.
\bibitem{Z}
Y.-C. Zhu:
\it{Modular invariance of characters of vertex operator algebras},
J. Amer. Math. Soc. \textbf{9} (1996), 237-302.
\end{thebibliography}
\end{document}